\documentclass[12pt]{article}
\usepackage{latexsym,amsmath,amssymb}
\usepackage[usenames]{color}

\newfont{\bb}{msbm10 at 12pt}
\def\t{\hbox{\bb T}}
\def\r{\hbox{\bb R}}

\def\h{\hbox{\bb H}}
\def\zz{\hbox{\bb Z}}

\def\s{\hbox{\bb S}}

\def\amb{\mathcal{M}}
\def\man{\mathcal{N}}
\newcommand{\campo}{\mathfrak{X}}
\newcommand{\camb}{\overline{\nabla}}
\newcommand{\norm}[1]{\left\Vert #1 \right\Vert}
\newcommand{\abs}[1]{\left\vert #1 \right\vert}
\newcommand{\set}[1]{\left\{#1\right\}}
\newcommand{\meta}[2]{\langle #1,#2 \rangle }

\usepackage[latin1]{inputenc}
\usepackage {amsmath}
\usepackage {amsthm}
\usepackage{times}
\usepackage{amscd}
\usepackage{epsf}

\numberwithin{equation} {section}

\begin{document}

\theoremstyle{plain}\newtheorem{lemma}{Lemma}[section]
\theoremstyle{plain}\newtheorem{definition}{Definition}[section]
\theoremstyle{plain}\newtheorem{theorem}{Theorem}[section]
\theoremstyle{plain}\newtheorem{proposition}{Proposition}[section]
\theoremstyle{plain}\newtheorem{remark}{Remark}[section]
\theoremstyle{plain}\newtheorem{corollary}{Corollary}[section]

\begin{center}
\rule{15cm}{1.5pt} \vspace{.6cm}

{\Large \bf Area estimates and rigidity of capillary\\[3mm] $H-$surfaces in three-manifolds with boundary} \vspace{0.4cm}

\vspace{0.5cm}

{\large Jos$\acute{\text{e}}$ M. Espinar$\mbox{}^\dag$ and Harold Rosenberg$\mbox{}^\ddag$}\\
\vspace{0.3cm} \rule{15cm}{1.5pt}
\end{center}

\vspace{.5cm}

\noindent $\mbox{}^\dag$ Instituto de Matematica Pura y Aplicada, 110 Estrada Dona
Castorina, Rio de Janeiro 22460-320, Brazil; e-mail: jespinar@impa.br

\noindent $\mbox{}^\ddag$ Instituto de Matematica Pura y Aplicada, 110 Estrada Dona
Castorina, Rio de Janeiro 22460-320, Brazil; e-mail: rosen@impa.br

\vspace{.3cm}

\begin{abstract}
We obtain a bound for the area of a capillary $H-$surface in a three-manifold with umbilic boundary and controlled sectional curvature. We then analyze the geometry when this area bound is realized, and obtain rigidity theorems. As a side product, we obtain existence of totally geodesic embedded surfaces in hyperbolic three-manifolds under the assumption of the existence of a $H-$surface realizing the area bound, in particular, under the existence of a totally umbilic $H-$surface.  
\end{abstract}

{\bf Key Words:} Capillary surfaces; Rigidity of three-manifold with boundary; hyperbolic manifold with boundary.

\section{Introduction}

In a recent paper, L. Mazet and H. Rosenberg \cite{LMazHRos} showed that a minimal two sphere $\Sigma$ satisfies a lower area bound when immersed in a Riemannian manifold of sectional curvatures bounded between $0$ and $1$. Moreover, in the case that equality is achieved, they were able to characterize the Riemannian manifold and the two sphere as either the standard three sphere $\s ^3$ with $\Sigma$ totally geodesic or a quotient of the product manifold $\s ^2 \times \r $ endowed with the standard product metric with $\Sigma$ and $\s^2 \times \set{t}$. These results were motivated by Klingerberg's length estimate for closed geodesics in Riemannian surfaces of curvature between $0$ and $1$, and Calabi's characterization theorem when the equality is achieved (cf. \cite{WKli}). The authors \cite{LMazHRos} also proved a rigidity result for hyperbolic cusps. 

In this note, we will extend the area estimate and rigidity results to capillary surfaces of constant mean curvature immersed in a Riemnnian manifold with umbilic boundary and controlled sectional curvatures. In the past few years, other area estimates have been established under different hypothesis on the surfaces, for example free boundary surfaces, or in the Riemannian manifolds, for example a bound on the scalar curvature. We invite the reader to see the papers of L. Ambrozio \cite{LAmb}, S. Brendle \cite{SBre}, A. Fraser and R. Schoen \cite{AFraRSch} and F. Marques and A. Neves \cite{FMarANev}.

\section{Notation}

First, we will fix the notation we will use along the paper. 

Let $(\amb  , \partial \amb )$ be a complete Riemannian $3-$manifold with boundary and Riemannian metric $\meta{\cdot}{\cdot}$. Denote by $\camb$ the Levi-Civita connection associated to $\meta{\cdot}{\cdot}$ in $\amb $, $K_{sect}$ its sectional curvatures and ${\rm Ric}$ the Ricci tensor.

Let $\eta$ be the inward unit normal along $\partial \amb$. We say that $\partial \amb$ is {\it umbilic}, with {\it umbilicity factor} $\alpha \in \r$, if $II_{\partial \amb} = \alpha \, I _{\partial \amb}$, where $I_{\partial \amb}$ and $II_{\partial \amb}$ are the first and second fundamental forms of $\partial \amb$ in $\amb$. Here, the second fundamental form is computed with respect to the inward orientation $\eta$, i.e., 
$$ II_{\partial \amb} (X,Y) = -\meta{\camb _X \eta}{Y} , \, \, X,Y \in \campo (\partial \amb) ,$$in particular, $\partial \amb$ is totally geodesic if $\alpha = 0$. Along this paper, $\eta$ will always stand by the inward normal along $\partial \amb$.

Now, let $\Sigma $ be an oriented compact surface with boundary $\partial \Sigma$ and unit normal $N$, $N$ chosen so that $\norm{\vec{H}} N = \vec{H}$ when $\vec{H}\neq 0$; $\vec{H}=2HN$, where $\vec{H}$ and $H$ are the mean curvature vector and mean curvature function respectively. If $H$ is constant along $\Sigma$, we say that $\Sigma$ is a {\it $H-$surface}.

Also, denote by $II_\Sigma$ the second fundamental form of $\Sigma$ in $\amb $ with respect to $N$ and by $K_e$ and $K_\Sigma$ its extrinsic and Gaussian curvature, the extrinsic curvature $K_e$ is nothing but the product of the principal curvatures. Associated to the mean and extrinsic curvatures one can define the {\it skew curvature} as $\Phi = \sqrt{H^2 -K_e}$ that measures how far the surface is from being umbilic. Throughout this paper, we will denote by $\abs{\Sigma}$ and $\abs{\partial \Sigma}$ the area of $\Sigma$ and the length of $\partial \Sigma$ respectively.

We assume that $\Sigma \subset \amb$ and $\partial \Sigma \subset \partial \amb $. We say that $\Sigma $ is a {\it capillary surface of angle $\beta$} in $\amb$ if the outer conormal $\nu$ along $\partial \Sigma$ and the unit normal along $\partial \amb$ make a constant angle $\beta $ along $\partial \Sigma$, i.e., there exists a constant $\beta \in [0, \pi/2 )$ so that $\meta{\nu}{\eta} = -\cos \beta $. In particular, when $\beta =0$ or, equivalently, $\Sigma $ meets orthogonally $\partial \amb$, we say that $\Sigma$ is a {\it free boundary surface}.

\section{Outline of the paper: results and examples}

We obtain a lower bound for the area of a capillary $H-$disk immersed in a Riemannian $3-$manifold of sectional curvatures between $0$ and $1$. We prove

\begin{quote}
{\bf Lemma \ref{Lem:Area}.} {\it Let $(\amb,\partial \amb) $ be a complete orientable Riemannian $3-$manifold with boundary. Let $\Sigma \subset \amb$ be a compact oriented $H-$disk ($H\geq 0$) with boundary $\partial \Sigma \subset \partial\amb$. Assume that 
\begin{itemize} 
\item The sectional curvatures of $\amb $ satisfy $ K_{sect} \leq 1 $,
\item  $\partial \amb $ is umbilic, with umbilicity factor $\alpha \in \r $,
\item $\Sigma$ is a capillary disk of angle $\beta \in [0, \pi /2)$,
\end{itemize}then, 
\begin{equation*}
2\pi  \leq (1+H^2)\abs{\Sigma} + \frac{\alpha+ (H+{\rm max}_{\partial \Sigma}\Phi)\sin\beta }{\cos \beta} \abs{\partial \Sigma} .
\end{equation*}

Moreover, equality holds if, and only if, $\Sigma$ is umbilic, $K_\Sigma = 1 + H^2$ and $K_{sect} \equiv 1$ along $\Sigma$.}
\end{quote}

When $\Sigma$ is a free boundary disk and $\partial \amb $ is totally geodesic, the above inequality reads as 
$$ 2\pi  \leq (1+H^2)\abs{\Sigma}  $$and equality holds if, and only if, $\Sigma$ is umbilic, $K_\Sigma = 1 + H^2$ and $K_{sect} \equiv 1$ along $\Sigma$.

Now, let us describe the model cases of Riemannian manifolds $(\amb , \partial \amb)$ where free boundary disks achieve the equality. In this case we have two different models:

\begin{itemize}
\item {\bf Model 1:} Let $\s ^3 \subset \r ^4$ be the standard unit three-sphere embedded in the four dimensional Euclidean Space with the standard Euclidean metric $\meta{\cdot}{\cdot}_0$. Then, the upper hemisphere, given by
$$ \s ^3 _+ := \set{ (x_1,x_2,x_3,x_4) \in \s ^3 \, : \, x_4 \geq 0} ,$$is a complete manifold with constant sectional curvatures equal to $1$ and totally geodesic boundary 
$$ \partial \s ^3 _+ =\set{ (x_1,x_2,x_3,x_4) \in \s ^3 \, : \, x_4 = 0} ,$$which is isometric to a two-sphere $\s ^2$. From now on, we denote by $\s ^n (r)$ the standard $n-$dimensional sphere of constant sectional curvatures $1/r^2$.

For any $x \in\partial \s ^3 _+ $, let $\mathcal B _x(R)$ be the geodesic ball of the standard three-sphere $\s ^3$ centered at $x$ of radius $R$. Fix $ H \geq 0  $ a constant and set $R_H = \frac{\pi}{2}-{\rm arctg} \, H$. Then, we have 
$$ D_H := \partial \mathcal B_x (R_H) \cap \s ^3 _+ $$is a umbilic $H-$disk orthogonal to $\partial \s ^3 _+$ such that 
$$ \abs{D_H} = \frac{2\pi}{1+H^2}. $$

\item {\bf Model 2:} Let $\s ^2 _+$ be the upper hemisphere of the standard $2-$sphere, i.e., 
$$ \s ^2 _+ := \set{ (x_1,x_2,x_3) \in \s ^2 \, : \, x_3 \geq 0} .$$

Clearly, $\r \times \s^2 _+ $ endowed with the standard product metric is a complete manifold with sectional curvatures between $0$ and $1$ and totally geodesic boundary 
$$ \partial \left( \r \times \s^2 _+\right) = \r \times \partial \s^2 _+ .$$

For any $t \in \r$, $D_t = \set{t} \times \s ^2 _+$ is a totally geodesic minimal disk orthogonal to $\r \times \partial \s^2 _+$ such that 
$$ \abs{D_t} = 2\pi .$$
\end{itemize}

The conclusion of the rigidity result says that, under the existence of a free boundary $H-$disk achieving the equality, the manifold $\amb $ must be locally isometric to one of the models above. 

\begin{quote}
{\bf Theorem \ref{Th:Rigidity}.} {\it Let $(\amb,\partial \amb) $ be a complete orientable Riemannian $3-$manifold with boundary. Assume that $\amb$ has sectional curvatures $0 \leq K_{sect} \leq 1$ and $\partial \amb$ is connected and totally geodesic. If there exists a compact oriented embedded $H-$disk $\Sigma \subset \amb$ orthogonal to $\partial \amb$ such that 
$$ \abs{\Sigma} = \frac{2\pi}{1+H^2} .$$

Then: 
\begin{itemize}
\item If $H>0$, the mean convex side of $\amb \setminus \Sigma$, call this component $\mathcal U$, is isometric to $\mathcal B _x (R_H) \cap \s ^3 _+ \subset \s ^3$ with the standard metric, $R_H = \frac{\pi}{2}-{\rm arctg} \, H$ and $x \in \partial \s ^3_+$. Moreover, $\Sigma$ is a disk $D_H$ described in the Model 1.

\item If $H=0$, $\amb$ is isometric either to $\s ^3 _+$ with its standard metric of constant sectional curvature one, or a quotient of $\r \times \s ^2 _+$ with the standard product metric. Moreover, $\Sigma$ is a disk $D_H$ or $D_t$ as described in the Models 1 and 2.
\end{itemize}}
\end{quote}



Section \ref{SectNeg} is devoted to capillary $H-$surfaces immersed in a complete manifold with sectional curvatures less or equal to $-1$ and umbilic boundary. We first obtain an area bound.

\begin{quote}
{\bf Lemma \ref{Lem:Area2}.} {\it Let $(\amb,\partial \amb) $ be a complete orientable Riemannian $3-$manifold with boundary. Let $\Sigma \subset \amb$ be a compact oriented $H-$surface ($H\geq 0$) with boundary $\partial \Sigma \subset \partial\amb$. Assume that 
\begin{itemize} 
\item The sectional curvatures of $\amb $ satisfy $ K_{sect} \leq -1 $,
\item  $\partial \amb $ is umbilic, with umbilicity factor $\alpha \in \r $,
\item $\Sigma$ is a capillary surface of angle $\beta _i \in [0,\pi/2)$ at each connected component $\partial \Sigma _i$ of the boundary $\partial \Sigma$,
\end{itemize}then, 
\begin{equation*}
-2\pi \chi (\Sigma) + \sum_{i=1}^k\frac{\alpha+ \sin\beta _i (H+{\rm max}_{\partial \Sigma}\Phi))}{\cos \beta _i} \abs{\partial \Sigma _i}\geq ( 1- H^2 )\abs{\Sigma} ,
\end{equation*}where $k$ is the number of connected components of $\partial\Sigma$. Moreover, equality holds if, and only if, $\Sigma$ is umbilic, $K_\Sigma =  H^2-1$ and $K_{sect} \equiv -1$ along $\Sigma$.}
\end{quote}

As we did in Section \ref{SectPos}, we will describe the model cases when the equality is achieved by a free boundary $H-$surface, $H \in [0, 1]$, in a complete manifold with $K_{sect} \leq -1$ and totally geodesic boundary. 

In both cases, $H=1$ and $H \in [0,1)$, $\amb$ will be a hyperbolic manifold with totally geodesic boundary and $\Sigma$ will be a constant intrinsic curvature umbilical $H-$surface orthogonal to $\partial \amb$.

\begin{itemize}
\item {\bf Model 3:} We begin with the case $H \in [0,1)$. Let $S$ be a closed oriented surface of curvature $-1$. $S$ is a quotient of $\h ^2$ by a cocompact Fuchsian group $\Gamma$ of isometries of $\h^2$. Consider $P_0 \equiv \h ^2$ as isometrically embbeded in $\h ^3$ as a totally geodesic plane, with $n$ a unit normal vector field to $P_0$ in $\h ^3$. 

We parametrize $\h ^3$ in Fermi coordinates by $ F : \r \times P_0 \to \h ^3 $, where 
$$ F(t,x) := {\rm exp}_x (t n(x)) ,$$here ${\rm exp}$ denotes the exponential map in $\h ^3$. The metric of $\h ^3$ in these coordinates is $dt ^2 + \cosh^2 (t) g_{-1} $, where $g_{-1}$ is the standard hyperbolic metric of curvature $-1$. Observe that 
$$ P_t := \set{ F(t,x) \, : \, x \in H_0} $$is an equidistant surface of $P_0$ of constant mean curvature $ {\rm tgh} (t)$. 

The group $\Gamma$ extends to isometries of $\h ^3$: for $\gamma \in \Gamma $, $t \in \r $ and $x \in H_0$, define 
$$ \gamma (F(t,x)) = F(t,\gamma (x)) .$$

The extended action leaves each $P_t$ invariant and $P_t / \Gamma = S_t$ is a constant sectional curvature umbilical surface in $\h ^3 /\Gamma $. $ \man \equiv \h^3 / \Gamma $ is homeomorphic to $\r \times S$.

Let $\beta $ be a simple closed geodesic of $S$; $\beta $ lifts to a complete geodesic $\tilde \beta$ of $P_0$. $F(\r \times \tilde \beta)$ is a totally geodesic plane in $\h ^3$ which quotients to a totally geodesic annulus $A(\beta)$ in $\man $.

Let $\beta _1 . \ldots , \beta _k$, where $k$ is the genus of $S$, be pairwise disjoint simple closed geodesics of $S$, separating $S$ in two components. Let $\Sigma _0$ be one of the components. $\Sigma _0$ lifts to $\tilde \Sigma _0 \subset P_0$, $\tilde \Sigma _0$ bounded by $\bigcup _{i=1}^k \tilde \beta _i$ . Note that $F_t (\tilde \Sigma _0)= F(t, \tilde \Sigma _0 ) \subset P_t$ is a constant intrinsic curvature surface, of constant mean curvature $ {\rm tgh} (t)$. The boundary of $F_t (\tilde \Sigma _0)$ meets the totally geodesic planes $F(\r \times \tilde \beta _i)$ orthogonally.

The quotient of $M=\set{F_t (\tilde \Sigma _0) \, : \, t \in \r }$ is then a hyperbolic three manifold $\amb = M/\Gamma$, of constant sectional curvatures $-1$, bounded by totally geodesic annuli $ \partial \amb = \bigcup _{i=1}^k  A(\beta _i)$. $\amb$ is foliated by the equidistant surfaces $\Sigma _t$, the quotient of $F_t (\tilde \Sigma _0)$. Each $\Sigma _t$ meets $\partial \amb$ orthogonally. Moreover, the mean curvature of $\Sigma _t$ is given by $H = {\rm tgh } (t)$, $t \geq 0$. 

Observe that the totally umbilic surface $\Sigma _t$, $t\geq 0$, separates $\man / \Gamma $ into two components, say $\mathcal H _t$ is the mean convex side. Note that the mean curvature vector along $\Sigma _t$ points into the component where $\Sigma _0$ is contained. When $t=0$ and $\Sigma _0$ is minimal, both components are isometric and we denote the component containing $\Sigma _t$, for $t>0$, as $\mathcal H _0$. Finally, we will denote by $\mathcal H (t) = \mathcal H_0 \cap \mathcal H _t$, for all $t >0$. Observe that $\mathcal H (t)$ is nothing but the connected component of $\amb \setminus \left( \Sigma _0 \cup \Sigma _t\right)$ whose boundary is mean convex.

\item {\bf Model 4:} Finally, we describe the case $H=1$. Consider a horosphere $H_{0}$ of $\h ^{3}$ and a $\zz ^{2}$ group $\Gamma$ of parabolic isometries leaving $H_{0}$ invariant. Let $n$ be a unit normal vector field to $H_{0}$ in $\h ^{3}$, pointing to the mean convex side of $H_{0}$. $\h ^{3}$ has the Fermi coordinates $ F : \r \times H_0 \to \h ^3 $, where 
$$ F(t,x) := {\rm exp}_x (t n(x)) ,$$here ${\rm exp}$ denotes the exponential map in $\h ^3$. The metric of $\h ^3$ in these coordinates is $dt ^2 + e^{2t} g_{e} $, where $g_{e}$ is the standard Euclidean metric of curvature $0$.

As before, $\Gamma$ acts on $\h^{3}$, leaving each horosphere
$$ H_{t} = \set{ F(t,x) \, :\, x \in H_{0}} $$invariant. Each $H_{t}$ has mean curvature $1$. $\man = \h^{3}/ \Gamma$ is a hyperbolic three manifold of constant sectional curvatures $-1$ isometric to $\r \times \t ^{2}$ with the metric $dt^{2 } + e^{-2t} g_{e}$, where $g_{e}$ is the standard flat metric of $\t ^{2} \equiv H_{0}/\Gamma $, a $2-$torus. 

Let $\beta _{1}$ and $\beta _{2}$ be disjoint simple closed geodesics of $\t^{2}$, and let $\Sigma _{0}$ be one of the annuli of $\t ^{2}$ bounded by $\beta _{1} \cup \beta _{2}$. $\Sigma _0$ lifts to $\tilde \Sigma _0 \subset H_0$, $\tilde \Sigma _0$ bounded by $ \tilde \beta _1 \cup \tilde \beta _{2}$, where $\tilde \beta _{i}$ is the lift of $\beta _{i}$ to $H_{0}$, a horocycle for each $i=1,2$. Note that $F_t (\tilde \Sigma _0)= F(t, \tilde \Sigma _0 ) \subset H_t$ is a flat surface, of constant mean curvature $1$. The boundary of $F_t (\tilde \Sigma _0)$ meets the totally geodesic planes $F(\r \times \tilde \beta _i)$ orthogonally, $i=1,2$.

The quotient of $M=\set{F_t (\tilde \Sigma _0) \, : \, t \in \r }$ is then a hyperbolic three manifold $\amb = M/\Gamma$ with totally geodesic boundary $\partial \amb = \bigcup_{i=1}^{2}F(\r \times \beta _{i}) / \Gamma$. $\amb$ is foliated by the equidistant surfaces $\Sigma _t$, the quotient of $F_t (\tilde \Sigma _0)$. Each $\Sigma _t$ meets $\partial \amb$ orthogonally. Each $\Sigma _{t} = F(t, \tilde \Sigma _{0})/\Gamma$ is a flat constant mean curvature one annulus meeting $\partial \amb$ orthogonally. 

Observe that the flat torus $\Sigma _0$ separates $\man $ into two components. We denote by $\mathcal P$ the mean convex side of $\Sigma$, which is parabolic.

\end{itemize}

The conclusion of the rigidity result in this case is quite surprising under the existence of a free boundary $H-$surface $\Sigma$, $H \in (0, 1)$, achieving the equality in the area estimate. It says that we can find an embedded totally geodesic minimal surface $\Sigma _m$ orthogonal to $\partial \amb$ in the mean convex side of $\amb \setminus \Sigma $ such that a domain in $\amb $ must be isometric to the model $\mathcal H (t)$ constructed above, whose boundary components are isometric to $\Sigma$ and $\Sigma _{m}$. When $H=1$, the mean convex side of $\amb \setminus \Sigma$ is isometric to the model $\mathcal P$, and $\Sigma$ is one of the $\Sigma _t$, as expected. Specifically,

\begin{quote}
{\bf Theorem \ref{Th:Rigidity2}.} {\it Let $(\amb,\partial \amb) $ be a complete orientable Riemannian $3-$manifold with boundary. Assume that $\amb$ has sectional curvatures $K_{sect} \leq -1$ and $\partial \amb$ is totally geodesic. Assume that there exists a compact oriented embedded $H-$surface $\Sigma \subset \amb$, $H \in (0, 1]$, orthogonal to $\partial \amb$ with non-positive Euler characteristic. Then:
\begin{itemize}
\item If $H \in (0,1)$, $\Sigma$ separates and $ \abs{\Sigma} = \dfrac{2\pi \abs{\chi(\Sigma)}}{1- H^2}$, $\chi(\Sigma ) < 0$, then there exists a totally geodesic minimal surface $\Sigma _m$ orthogonal to $\partial \amb$ and an isometry 
$$ F: \left( [0,  {\rm arctgh } \, H) \times \Sigma ,  dt^2 + \cosh ^2 \left( t \right)g_{-1} \right) \to \amb ,$$where $g_{-1}$ denotes the standard metric of constant curvature $-1$, such that 
\begin{itemize}
\item $F (0, \Sigma) = \Sigma $ and $F( {\rm arctgh } \, H , \Sigma) = \Sigma _{m}$, and
\item $ F (t, \Sigma ) = \Sigma _{t}$ is an embedded totally umbilic $H-$surface, $H={\rm tgh}( {\rm arctgh } \, H -t)$, orthogonal to $\partial \amb$ for all $t \in [0 ,  {\rm arctgh } \, H)$.
\end{itemize}

Moreover, if $\Sigma _m$ is non-orientable, a tubular neighborhood of $\Sigma _m$ is foliated by its equidistants $\Sigma _t$.




\item If $H=1$ and  $\chi (\Sigma) =0$, the mean convex side of $\amb \setminus \Sigma $, call this component $\mathcal U$, is isometric to $ [0, + \infty) \times \Sigma$ endowed with the product metric $g = dt^2 + e^{-2t}g_{e}$, where $g_e$ is the standard Euclidean metric of curvature $0$. That is, $\mathcal U$ is isometric to a cusp hyperbolic end $\mathcal P$ corresponding to Model 4 and $\Sigma$ is a slice.

\end{itemize}}
\end{quote}

The results above extend to compact $H-$surfaces without boundary. The area estimate takes the following form

\begin{quote}
{\bf Lemma \ref{Lem:AreaCompact}.} {\it Let $\amb $ be a complete orientable Riemannian $3-$manifold with sectional curvatures $ K_{sect} \leq -1 $. Let $\Sigma \subset \amb$ be a closed oriented $H-$surface, $H \in [0,1)$, of genus $g(\Sigma ) \geq 2$. Then, 
\begin{equation*}
4\pi (g(\Sigma) -1) \geq ( 1- H^2 )\abs{\Sigma} .
\end{equation*}

Moreover, equality holds if, and only if, $\Sigma$ is umbilic, $K_\Sigma =  H^2-1$ and $K_{sect} \equiv -1$ along $\Sigma$.}
\end{quote}

As a first consequence of the above area estimate we have that, when $H^{2} <1$, $\Sigma$ has genus at least two. Observe that Model 3 has a counterpart without boundary. Let $S$ be the quotient of $\h ^2$ by a cocompact Fuchsian group $\Gamma$ of isometries of $\h^2$. Consider Fermi coordinates $F: \r \times P_0 \to \h ^3$, where $P_0 \equiv \h^2$ is a isometrically embbeded totally geodesic plane in $\h ^3$. Then, 

\begin{itemize}
\item {\bf Model 5:} $ \amb = \set{F_t (P_0) \, : \, t \in \r } / \Gamma $ is a complete hyperbolic three-manifold without boundary.  One can see that $\amb$ is foliated by the equidistant surfaces $\Sigma _t$, the quotient of $F_t (\tilde \Sigma _0)$. Each $\Sigma _{t}$ is a constant intrinsic curvature surface, of constant mean curvature $ {\rm tgh} (t)$.

Again, $\Sigma _t$, which is a totally umbilic surface, separates $\amb $ into two components. As we did in Model 3, we will denote by $\widetilde{ \mathcal H} (t)$, $t>0$, the connected component of $\amb \setminus \left( \Sigma _0 \cup \Sigma _t \right)$ whose boundary is mean convex. 
\end{itemize}

In the case we get the equality in the area estimate above, we can obtain the existence of an embedded totally geodesic surface. 

\begin{quote}
{\bf Theorem \ref{Th:RigidityCompact}.} {\it Let $\amb $ be a complete orientable Riemannian $3-$manifold with sectional curvatures $K_{sect} \leq -1$. Assume that there exists a compact oriented embedded $H-$surface $\Sigma \subset \amb$, $H \in (0, 1)$, $\Sigma$ separates $\amb$, with negative Euler characteristic such that 
$$ \abs{\Sigma} = \dfrac{2\pi \abs{\chi(\Sigma)}}{1- H^2} ,$$then there exists an embedded totally geodesic minimal surface $\Sigma _m$ and an isometry 
$$ F: \left( [0,  {\rm arctgh } \, H] \times \Sigma ,  dt^2 + \cosh ^2 \left( t \right)g_{-1} \right) \to \amb ,$$where $g_{-1}$ denotes the standard metric of constant curvature $-1$, such that 
\begin{itemize}
\item $F (0, \Sigma) = \Sigma $ and $F( {\rm arctgh } \, H , \Sigma) = \Sigma _{m}$, and
\item $ F (t, \Sigma ) = \Sigma _{t}$ is an embedded totally umbilic $H-$surface, $H={\rm tgh}( {\rm arctgh } \, H -t)$, for all $t \in (0 ,  {\rm arctgh } \, H)$.
\end{itemize}

Moreover, if $\Sigma _m$ is non-orientable, a tubular neighborhood of $\Sigma _m$ is foliated by its equidistants $\Sigma _t$.}



\end{quote}

\begin{remark}
The case $H=1$ was done by L. Mazet and H. Rosenberg in \cite{LMazHRos}.
\end{remark}

It is remarkable that the minimal case can not be classified. Observe that any embedded totally geodesic surface in a hyperbolic three-manifold satisfy the area estimate given above and, at a first glance, we have no chance to know the structure of the three-manifold. Nevertheless, the above result says that the existence of a totally umbilic $H-$surface $\Sigma$ embedded in a hyperbolic three-manifold determines the mean convex side up to we reach a totally geodesic surface.

\section{Manifolds with umbilic boundary and $0 \leq K_{sect} \leq 1$}\label{SectPos}

In this section we will study capillary $H-$disks immersed in a three-manifold with umbilic boundary and sectional curvatures between $0$ and $1$. First, we obtain a lower bound for the area of such a capillary $H-$disk. Second, we obtain rigidity of the three-manifold with boundary assuming that the $H-$disk is orthogonal to the boundary of the three-manifold and achieves the bound for the area. In the latter case, we assume the boundary of the three-manifold is totally geodesic.

\subsection{Area estimate}

In this section, we will obtain a lower bound for the area of a capillary $H-$disk immersed in a Riemannian $3-$manifold whose boundary is contained in the boundary of the manifold. Such a lower bound will depend on the sectional curvatures of the $3-$manifold, the second fundamental form of the boundary of the $3-$manifold, the mean curvature of the $H-$disk and the length of the boundary of the disk. Let us now be more precise.

\begin{lemma}\label{Lem:Area}
Let $(\amb,\partial \amb) $ be a complete orientable Riemannian $3-$manifold with boundary. Let $\Sigma \subset \amb$ be a compact oriented $H-$disk ($H\geq 0$) with boundary $\partial \Sigma \subset \partial\amb$. Assume that 
\begin{itemize} 
\item The sectional curvatures of $\amb $ satisfy $ K_{sect} \leq 1 $,
\item  $\partial \amb $ is umbilic, with umbilicity factor $\alpha \in \r $,
\item $\Sigma$ is a capillary disk of angle $\beta \in [0, \pi /2)$,
\end{itemize}then, 
\begin{equation*}
2\pi  \leq (1+H^2)\abs{\Sigma} + \frac{\alpha+ (H+{\rm max}_{\partial \Sigma}\Phi)\sin\beta }{\cos \beta} \abs{\partial \Sigma} .
\end{equation*}

Moreover, equality holds if, and only if, $\Sigma$ is umbilic, $K_\Sigma = 1 + H^2$ and $K_{sect} \equiv 1$ along $\Sigma$.
\end{lemma}

\begin{proof}
First, by the Gauss equation and the Arithmetic-Geometric Inequality, we obtain
$$ K _{\Sigma} = K_{e}+ K_{sect} \leq H^{2} +1,$$hence, integrating over $\Sigma$, the Gauss-Bonnet Formula yields
\begin{equation}\label{Eq1}
2\pi = \int_{\Sigma} K_{\Sigma} + \int_{ \partial \Sigma} k_{g} \leq (1+H^{2})\abs{\Sigma} + \int _{\partial \Sigma} k_{g} .
\end{equation}

Second, let $t$ denote a unit tangent vector field along $\partial \Sigma$, clearly, $t \in \campo (\partial \amb)$, and let $n = J t$ the rotation of angle $\pi/2$ on $\partial \amb$.

On the one hand, since $\set{t,n}$ is an orthonormal frame along $\partial \Sigma$, we have
$$ 2 H_{\partial \amb} = -\meta{t}{\camb _t \eta} - \meta{n}{\camb _n \eta} ,$$and since $\partial \amb$ is umbilic, we obtain
$$ -\meta{t}{\camb _t \eta} = \alpha . $$

On the other hand, by the capillary condition, $ -\eta = \cos \beta \, \nu + \sin \beta \,  N $ along $\partial \Sigma$, hence 
$$ -\meta{t}{\camb _t \eta} = \cos \beta \meta{t}{\camb _t \nu} + \sin \beta \meta{t}{\camb _t N} = \cos \beta \, k_g - \sin \beta \, II_\Sigma (t , t ) .$$

Therefore, combining both equations we get
\begin{equation}\label{Eq2}
\cos \beta \,  k_g  = \alpha + \sin \beta \, II_\Sigma (t, t) \text{ along } \partial \Sigma .
\end{equation}

Finally, since $\Sigma$ has constant mean curvature $H$, we have that $II_\Sigma (t, t) \leq H + \Phi $ and hence, substituting \eqref{Eq2} into \eqref{Eq1} yields
\begin{equation}\label{Eq3}
2\pi \cos \beta \leq (1+H^2)\cos \beta \abs{\Sigma} + (\alpha+ (H+{\rm max}_{\partial \Sigma}\Phi) \sin\beta  ) \abs{\partial \Sigma} ,
\end{equation}as claimed.

Moreover, equality holds if and only if, equality holds in \eqref{Eq1}, that is, $K_\Sigma = H^2 +1$, which implies that $\Sigma $ is umbilic and $K_{sect} \equiv 1$ along $\Sigma$. From the Gauss Equation, we deduce $K_\Sigma = 1 + H^2$.
\end{proof}

\subsection{Rigidity}

Now, we will characterize Riemannian manifolds $(\amb , \partial \amb)$ with totally geodesic boundary  and sectional curvatures $0 \leq K_{sect} \leq 1$ assuming the existence of an oriented $H-$disk meeting $\partial \amb$ orthogonally and of least area. According to Lemma \ref{Lem:Area}, if $\Sigma \subset \amb$ is an oriented $H-$disk  such that $\partial \Sigma \subset \partial \amb$ and meets $\partial \amb$ orthogonally, then 
$$  \abs{\Sigma} \geq \frac{2\pi}{1+H^2} ,$$where we are assuming that $\partial \amb$ is totally geodesic in $\amb$. Hence, we can announce

\begin{theorem}\label{Th:Rigidity}
Let $(\amb,\partial \amb) $ be a complete orientable Riemannian $3-$manifold with boundary. Assume that $\amb$ has sectional curvatures $0 \leq K_{sect} \leq 1$ and $\partial \amb$ is connected and totally geodesic. If there exists a compact oriented embedded $H-$disk $\Sigma \subset \amb$ orthogonal to $\partial \amb$ such that 
$$ \abs{\Sigma} = \frac{2\pi}{1+H^2} .$$

Then: 
\begin{itemize}
\item If $H>0$, the mean convex side of $\amb \setminus \Sigma$, call this component $\mathcal U$, is isometric to $\mathcal B _x (R_H) \cap \s ^3 _+ \subset \s ^3$ with the standard metric, $R_H = \frac{\pi}{2}-{\rm arctg} \, H$ and $x \in \partial \s ^3_+$. Moreover, $\Sigma$ is a disk $D_H$ described in the Model 1.

\item If $H=0$, $\amb$ is isometric either to $\s ^3 _+$ with its standard metric of constant sectional curvature one, or a quotient of $\r \times \s ^2 _+$ with the standard product metric. Moreover, $\Sigma$ is a disk $D_H$ or $D_t$ as described in the Models 1 and 2.
\end{itemize}

\end{theorem}
\begin{proof}
In order to get the rigidity result, we follow similar arguments to those included in \cite[Theorem 1]{LMazHRos}. Denote by $N$ the unit normal along $\Sigma$ pointing to the mean-convex side of $\amb \setminus \Sigma $, denote this component by $W$. In particular, if $\Sigma$ is minimal we can consider any component of $\amb \setminus \Sigma $. Consider the variation $ F : (-\epsilon , \epsilon) \times \Sigma \to \amb $ given by $ F_{t} (p) = {\rm exp}_{p}(t N(p)) $, where ${\rm exp}$ is the exponential map in $\amb$. Note that, since $\Sigma $ is orthogonal to $\amb$, the normal $N$ along the boundary $\partial \Sigma \subset \partial \amb$ is a tangent vector in $\partial \amb$ hence, since $\partial \amb$ is totally geodesic, for all $p\in \partial \Sigma$ we have
$$ F_{t}(p) \in \partial \amb  \text{ for all } t \in (-\epsilon , \epsilon) .$$

Set $\Sigma _{t} = F_{t} (\Sigma) \subset \amb $ and denote by $H_t$ and $\Phi _t := \sqrt{H_t ^2 - K_e ^t}$ the mean and skew curvature for all $t \in (-\epsilon , \epsilon)$, respectively. It is clear that for $\epsilon >0$ small enough, $\Sigma _t$ is an embedded disk such that $\Sigma _t \cap \partial \amb = \partial \Sigma _t$. Moreover, by the discussion above, $\Sigma _{t} \subset \amb$ and $\partial \Sigma _{t} \subset \partial \amb$. Also, by the Gauss Lemma, $\Sigma _{t } $ is orthogonal to $\amb$ for all $t \in (-\epsilon , \epsilon)$, $\epsilon$ small enough. Define
$$ \epsilon _0 := {\rm sup}\set{\epsilon >0 \, : \, \, F : [0,\epsilon)\times \Sigma  \to \amb  \text{ is embedded and } \Sigma _t  \cap \partial \amb = \partial \Sigma _t  } >0. $$

Let us pullback the metric $\meta{}{}$, via $F$, to $[0, \epsilon _0 ) \times \Sigma $, that is, $g = F^*\meta{}{}$ and 
$$ F : ( [0, \epsilon _0) \times \Sigma  ,g ) \to (\amb  , \meta{}{})$$ is a local isometry. Clearly, the pullback metric $g$ can be written as $g = g_t + dt^2$, where $g_t$ is a smooth family of metrics on $\Sigma $, in other words, $g_t$ is the induced metric of $\Sigma _t \subset \amb$. 

First, since $\Sigma _0 \equiv \Sigma$ achieves the equality in Lemma \ref{Lem:Area}, $(\Sigma _0 , g_0)$ is umbilic, $K_{sect}\equiv 1$ along $\Sigma _0$ and it isometric to a hemisphere of constant Gaussian curvature $1+H^2$ with its standard metric, i.e., $\Sigma \equiv \s ^2 _+\left(\frac{1}{\sqrt{1+H^2}}\right) $.

Second, by the Gauss-Bonnet formula and the Gauss Equation, we have 
\begin{equation*}
\begin{split}
2\pi &= \int _{\Sigma _{t}} K_{\Sigma _{t}} + \int _{\partial \Sigma} k_{g}^{t} \\
 & = \int _{\Sigma} (H_{t} +\Phi _{t})(H_{t}-\Phi _{t}) + \int_{\Sigma}K_{sect} + \int _{\partial \Sigma} k_{g}^{t} \\
 & \leq \abs{\Sigma _{t}} + \int _{\Sigma _{t}} H_t^{2} - \int _{\Sigma _{t}} \Phi _{t}^{2}  + \int _{\partial \Sigma} k_{g}^{t} ,
\end{split}
\end{equation*}and, bearing in mind that $\partial \amb$ is totally geodesic and $\Sigma _{t}$ orthogonal to $\partial \amb$, we get $ \int _{\partial \Sigma} k_{g}^{t} =0 $, which yields that 
\begin{equation}\label{Eq4}
\int _{\Sigma _{t}} \Phi _{t}^{2} \leq  \abs{\Sigma _{t}} - 2\pi + \int _{\Sigma _{t}} H_t^{2}  \text{ for all } t \in [0 , \epsilon _0) .
\end{equation}

For each $p \in \Sigma $, we denote $ \gamma _{p } $ the geodesic in $\amb$ with initial conditions $\gamma _{p}(0) = p$ and $\gamma _{p} ' (0) = N(p)$. Therefore, $F_{t} (p) = \gamma _{p} (t)$. Set 
$$ G(t) := \abs{\Sigma _{t}} - 2\pi + \int _{\Sigma _{t}} H_t^{2} \text{ for all } t \in  [0 , \epsilon _0) . $$


Then, by the First Variation Formula and the First Variation of the Mean Curvature
$$ \dfrac{\partial H_t}{\partial t} ={\rm Ric}(N_t) + 2 H^2_t + 2 \Phi _t ^2 \geq 0 \text{ for all } t \in [0 , \epsilon _0),$$and hence $h_p(t) = H_t (p)$ is a nondecreasing function of $t$ such that $h_p (t)\geq 0 $ for all $t \in [0  , \epsilon _0)$ and all $p \in \Sigma$. Thus, we obtain
\begin{equation*}
\begin{split}
G' (t) &= \int _{\Sigma _t} 2 H_t \left( \frac{\partial H_t}{\partial t} - H_t ^2 -1 \right) - \int_{\partial \Sigma _t} \meta{\gamma ' _p (t)}{\eta} \\
 & = \int _{\Sigma _t}  H_t \left( 2 \Phi _t ^2 + {\rm Ric}(N_t) -2 \right)\\
 &\leq  2 \int _{\Sigma _t}  H_t \Phi _t ^2
\end{split}
\end{equation*}where we have used that ${\rm Ric}(N_t) -2 \leq 0$ and $H_t \geq 0$, by the curvature hypothesis, and $\meta{\gamma ' _p (t)}{\eta} =0 $, since $\gamma ' _p (t) \in T_{\gamma _p (t)}\partial \amb$ and $\eta$ is normal along $\partial \amb$.

Therefore, since $\Sigma _t$ is a compact disk, for all $\bar \epsilon \in (0, \epsilon _0)$ there exists a constant $C >0$ (depending on $\bar \epsilon$) such that $\abs{H_t}\leq C $ for all $(t,p) \in [0  , \bar \epsilon) \times \Sigma  $. Thus, 
$$ G' (t) \leq 2 C G(t) \text{ for all } t \in [0  , \bar \epsilon) ,$$and, by Gronwall's Lemma, we get  
$$ G(t) \leq G(0)e^{2Ct} \text{ for all } t \in [0  , \bar \epsilon ) .$$

Since $\Sigma _0$ is totally umbillic, $(1+H^2)\abs{\Sigma _0} = 2\pi$ and $G(0)=0$, then
$$ G(t) \leq 0 \text{ for all } t \in [0  ,  \bar \epsilon ) .$$

But, $G(t) \geq 0 $ for all $t \in [0  , \bar \epsilon )$ by \eqref{Eq4} which implies that $G(t) \equiv 0$ for all $t \in [0  , \bar \epsilon )$. Continuing the argument, we obtain that $G(t) \equiv 0$ for all $t \in [0  ,  \epsilon _0)$.

Therefore, for all $t \in [0  , \epsilon _0)$ we have
\begin{enumerate}
\item[(a.1)] $\Sigma _t$ is umbilic,
\item[(a.2)] $H_t ({\rm Ric}(N_t) -2) =0$,
\item[(a.3)] $ 2 \dfrac{\partial H_t}{\partial t} = {\rm Ric}(N_t) + 2 H_t^2 $.
\end{enumerate}

We assert that if the mean curvature is positive at a point of an equidistant then it is positive at any point of this equidistant. More precisely,

\begin{quote}
{\bf Claim 1:} {\it Let $(p,t) \in (0 ,\epsilon _0)\times \Sigma $ be such that $H_t (p) >0$, then $H_t (q)>0$ for all $q \in \Sigma $.}
\end{quote}
\begin{proof}[Proof of Claim 1]
The proof is as in \cite[Claim 2]{LMazHRos}, but we include it here for the sake of completeness. Consider $t \in (0, \epsilon _0)$ and assume that $H_t (p) >0$ and set 
$$ \Omega  := \set{ q \in \Sigma  \, : \, H_t (q) >0} .$$

By continuity, $\Omega $ is a nonempty open subset of $\Sigma$. Let us see that $\Omega $ is closed. By continuity of the Ricci curvature and (a.2), we have that ${\rm Ric}_{(t,q)}(N_t) =2$ for all $q \in \overline{\Omega}$. Thus, item (a.3) implies that 
$$  \dfrac{\partial H_t}{\partial t} \geq 2 >0 \text{ for all } q \in \overline{\Omega},$$and therefore $H_t (q) > H \geq 0$, in other words, $q \in \Omega$. Thus, $\Omega$ is open and closed in $\Sigma$ so $\Omega \equiv \Sigma$ as claimed.
\end{proof}

Now, we will consider two cases, either $H=0$ or $H\neq 0$.

\begin{quote}
{\bf Case A: }{\it If  $H >0$, the mean convex side of $\amb \setminus \Sigma$, call this component $\mathcal U$, is isometric to the totally geodesic ball $\mathcal B _x (R_H) \cap \s ^3 _+ \subset \s ^3$ with the standard metric, $R_H = \frac{\pi}{2}-{\rm arctg} \, H$ and $x \in \partial \s ^3_+$. Moreover, $\Sigma$ is a disk $D_H$ described in the Model 1.}
\end{quote}
\begin{proof}[Proof of Case A]
In this case, by Claim 1, we obtain $H_t > 0$ for all $t \in [0 , \epsilon _0 )$ and hence ${\rm Ric}(N_t) = 2$, which implies that the sectional curvature of any two plane orthogonal to $\Sigma _t$ is one for all $t \in [0 , \epsilon _0 )$. 

On the one hand, item (a.3) gives
$$ \dfrac{\partial H_t}{\partial t} = 1 + H_t^2 \text{ for all } t \in [0 , \epsilon _0 ).$$

The above equation implies that 
$$ \dfrac{\partial }{\partial t} {\rm arctg} \, H_t = 1 \text{ for all } t \in [0 , \epsilon _0 ),$$then, by the Fundamental Theorem of Calculus, fixing $t \in [0 ,\epsilon _0)$, we get
$$ {\rm arctg} \,H_t (p) =  {\rm arctg} \, H + t \text{ for all } p \in \Sigma ,$$that is, $\Sigma _t$ has constant mean curvature 
\begin{equation}\label{HtPos}
H_t ={\rm tg}\left( {\rm arctg}\,  H + t\right) \text{ for all } t \in [0 , \epsilon _0).
\end{equation}

On the other hand, by the First Variation of the Area, we have 
$$ \dfrac{\partial }{\partial t} \abs{\Sigma _t} = -2 \, {\rm tg}\left( {\rm arctg} \, H + t\right) \abs{\Sigma _t} \text{ for all } t \in [0 , \epsilon _0) , $$which implies that 
$$ \abs{\Sigma _t} ^{1/2}= \sqrt{1+ H^2}\cos \left( t + {\rm arctg} \, H \right) \abs{\Sigma}^{1/2} \text{ for all } t \in [0 , \epsilon _0) ,$$ or, in other words, 
\begin{equation}\label{AreaPos}
\abs{\Sigma _t} =  2\pi \cos^2 \left( t + {\rm arctg} \, H \right) \text{ for all } t \in [0 , \epsilon _0) . 
\end{equation}

Thus, the Gauss-Bonnet Formula, \eqref{HtPos} and \eqref{AreaPos}, we have 
\begin{equation*}
\begin{split}
2\pi &= \int _{\Sigma _t} K_{\Sigma _t}  \leq (1+ H_t^2) \abs{\Sigma _t} \\
  & =  2 \pi \left( 1 + {\rm tg}^2 (t+ {\rm arctg} \, H ) \right) \cos^2 \left( t + {\rm arctg} \, H \right) \\
 & = 2 \pi 
\end{split}
\end{equation*}which tells us that the sectional curvature of the two plane tangent to $\Sigma _t$ is equals to $1$ for all $t \in [0 , \epsilon _0)$, and therefore
$$K_{\Sigma _t} = 1+  {\rm tg}^2 (t+ {\rm arctg} \, H ) = \cos^{-2} \left( t + {\rm arctg} \, H \right) \text{ for all } t \in (\bar \delta , \epsilon _0) .$$

Hence, $\Sigma _t$ is isometric to $\s ^2 \left(\cos \left( t + {\rm arctg} \, H \right)\right)$ for each $t \in [0 , \epsilon _0)$. In this case, $\epsilon _0 = \frac{\pi}{2} -  {\rm arctg} \, H $ and $F_{\epsilon _0} (\Sigma ) $ is a point. Therefore, $ [0 ,\frac{\pi}{2} -  {\rm arctg} \, H ) \times \Sigma$ with the pullback metric 
$$g = dt^2 + \sin ^2 \left( t + {\rm arctg} H\right) g_{+1}$$ is isometric to $\overline{\mathcal B_x (R_H)} \cap \s ^3 _+ $, for $R_H = \frac{\pi}{2} -{\rm arctg} \, H $ and some point $F_{\frac{\pi}{2} -  {\rm arctg} \, H} (\Sigma ) := x \in \partial \s ^3 _+$. Here, $g_{+1}$ denotes the standard metric $\s ^2$ of constant curvature one. Denote by $I : \overline{\mathcal B_x (R_H)} \cap \s ^3 _+ \to [0,\frac{\pi}{2} -  {\rm arctg} \, H ) \times \Sigma $ such an isometry.


So, $\tilde F = F \circ I  :  \overline{\mathcal B_x (R_H)} \cap \s ^3 _+ \to \amb $ is a local isometry and, since $\mathcal B_p (R_H ) \cap \s ^3 _+$ is simply connected and $\tilde F$ is injective on $\Sigma$, then $\tilde F$ is a global isometry onto its image, that is, the mean convex side of $\amb \setminus \Sigma $ is isometric to $\mathcal B_p (R_H) \cap \s ^3 _+$ with the standard metric of constant curvatures one. Hence, Claim A is proved.
\end{proof}

Now, we prove:

\begin{quote}
{\bf Case B:} If $H =0$, then $\amb$ is isometric either to $\s ^3 _+$ with its standard metric of constant sectional curvature one, or to a quotient of $\r \times \s ^2 _+$ with the standard product metric.
\end{quote}

\begin{proof}[Proof of Case B]
Claim 1 and (a.2) yield that 

\begin{enumerate}
\item[(i)] either, $H_t>0$ and ${\rm Ric}(N_t) = 2$  on $[0, \epsilon_0) \times \Sigma$,

\item [(ii)] or, $H_t=0$ and ${\rm Ric}(N_t) = 0$  on $[0, \epsilon_0) \times \Sigma$.
\end{enumerate}

In case (i), arguing as in Case A, we can conclude that each side of $\amb \setminus \Sigma$ is isometric to $\mathcal B _y (\frac{\pi}{2}) \cap \s ^3 _+$, $y = \pm x \in \partial \s ^3_+$, with the standard metric. Hence, $\amb$ must be isometric to $\s ^3 _+$ with the standard metric.

In case (ii), the First Variation Formula says $\abs{\Sigma _t} = 2 \pi $ for all $t \in [ 0, \epsilon _0)$. Therefore, $\Sigma _t$ is totally geodesic and the Gaussian curvature equals to one. Thus, the sectional curvatures of any two plane orthogonal to $\Sigma _t$ are zero, and for the two plane tangent to $\Sigma _t$ is one. Thus, the pullback metric $g$ on $(0, \epsilon _0) \times \Sigma $ can be written as $g = dt ^2 + g_{+1}$.

Therefore, $\epsilon _0 = + \infty$ and $(0, + \infty) \times \Sigma$ endowed with the pullback metric is isometric to $ (0, +\infty) \times  \s ^2 _+ $ with the standard product metric. As above, since both sides are mean convex, $\r \times \Sigma$ is isometric to $\r \times \s ^2 _+$ with the standard product metric. Moreover, since $F : \r \times \s ^2 _+ \to \amb$ is a local isometry, it is a covering map. This concludes the proof of Case B.
\end{proof}

This concludes the proof of Theorem  \ref{Th:Rigidity}.
\end{proof}




\section{Manifolds with umbilic boundary and $K_{sect} \leq -1$}\label{SectNeg}

In this section we will study compact capillary $H-$surfaces of non-positive Euler characteristic immersed in a three-manifold with umbilic boundary and sectional curvatures less or equal to $-1$. First, we obtain an upper bound for the area of such capillary $H-$surfaces. Second, we obtain rigidity of the three-manifold with boundary assuming that the $H-$surface is orthogonal to the boundary of the three-manifold and achieves the bound for the area. In the latter case, we assume the boundary of the three-manifold is totally geodesic.

\subsection{Area Estimate}

The area estimate will follow the lines of Lemma \ref{Lem:Area}. 

\begin{lemma}\label{Lem:Area2}
Let $(\amb,\partial \amb) $ be a complete orientable Riemannian $3-$manifold with boundary. Let $\Sigma \subset \amb$ be a compact oriented $H-$surface ($H\geq 0$) with boundary $\partial \Sigma \subset \partial\amb$. Assume that 
\begin{itemize} 
\item The sectional curvatures of $\amb $ satisfy $ K_{sect} \leq -1 $,
\item  $\partial \amb $ is umbilic, with umbilicity factor $\alpha \in \r $,
\item $\Sigma$ is a capillary surface of angle $\beta _i \in [0,\pi/2)$ at each connected component $\partial \Sigma _i$ of the boundary $\partial \Sigma$,
\end{itemize}then, 
\begin{equation*}
-2\pi \chi (\Sigma) + \sum_{i=1}^k\frac{\alpha+ \sin\beta _i (H+{\rm max}_{\partial \Sigma}\Phi))}{\cos \beta _i} \abs{\partial \Sigma _i}\geq ( 1- H^2 )\abs{\Sigma} ,
\end{equation*}where $k$ is the number of connected components of $\partial\Sigma$. Moreover, equality holds if, and only if, $\Sigma$ is umbilic, $K_\Sigma =  H^2-1$ and $K_{sect} \equiv -1$ along $\Sigma$.
\end{lemma}
\begin{proof}
As we did in Lemma \ref{Lem:Area}, by the Gauss-Bonnet Formula and the hypothesis, we get
\begin{equation*}
\begin{split}
2\pi \chi (\Sigma) &= \int _\Sigma  K_\Sigma + \int _{\partial \Sigma} k_g \\
 &= \int _\Sigma H^2 - \int _\Sigma \Phi ^2 + \int _\Sigma K_{sect} + \sum_{i=1}^k \int _{\partial \Sigma _i} k_g \\
 & \leq (H^2 -1)\abs{\Sigma} + \sum_{i=1}^k\frac{\alpha+ \sin\beta _i (H+{\rm max}_{\partial \Sigma}\Phi))}{\cos \beta _i} \abs{\partial \Sigma _i},
\end{split}
\end{equation*}as claimed.
\end{proof}

The area estimate obtained above is still valid when $H >1$, however it does not give information. We will see here how to construct annulus of constant mean curvature $H>1$ orthogonal to the boundary in the Model 3. We will use the notation we have introduced there. 

Let $\gamma$ be the minimizing geodesic in $\Sigma _{0}$  joining two different boundary components $\beta _{i}$ and $\beta _{j}$. Then, $\gamma $ lifts to a geodesic $\tilde \gamma \subset \tilde \Sigma _{0 } \subset P_{0}$ that minimizes the distance between $\tilde \beta _{i}$ and $\tilde \beta _{j}$.

Let $T _{r} = \set{ q \in \h ^{3} \, : \, d(p, \tilde \tilde \gamma ) =r }$, where $d$ denotes the hyperbolic distance, that is, the points at distance $r$ in $\h ^{3}$. This surface is contained in $\amb$ with constant mean curvature bigger than one, depending on $r$, and orthogonal to the boundary $\partial \amb$.

\subsection{Rigidity}

Now, we will characterize Riemannian manifolds $(\amb , \partial \amb)$ with totally geodesic boundary and sectional curvatures $K_{sect}\leq -1$ assuming the existence of an oriented $H-$surface meeting $\partial \amb$ orthogonally and of greatest area. By Lemma \ref{Lem:Area2}, if $\Sigma \subset \amb$ is an oriented compact $H-$surface, $H^2 \leq 1$, such that $\partial \Sigma \subset \partial \amb$ and meets $\partial \amb$ orthogonally, then 
$$  (1-H^2)\abs{\Sigma} \leq 2\pi \abs{\chi(\Sigma)} ,$$where we are assuming that $\partial \amb$ is totally geodesic in $\amb$. Note that, in the case $\chi(\Sigma) =0$,  Lemma \ref{Lem:Area2} says that if there exists a $H-$surface $\Sigma$, $H^{2}\leq 1$, orthogonal to the boundary with $\chi(\Sigma)=0$ then $\Sigma$ has constant mean curvature $H=1$, it is umbilic, $K_\Sigma = 0$ and $K_{sect} \equiv -1$ along $\Sigma$, without any area information.

It is remarkable that we can not characterize the manifold when $\Sigma$ is minimal. We will explain this in more detail after the proof of the following:

\begin{theorem}\label{Th:Rigidity2}
Let $(\amb,\partial \amb) $ be a complete orientable Riemannian $3-$manifold with boundary. Assume that $\amb$ has sectional curvatures $K_{sect} \leq -1$ and $\partial \amb$ is totally geodesic. Assume that there exists a compact oriented embedded $H-$surface $\Sigma \subset \amb$, $H \in (0, 1]$, orthogonal to $\partial \amb$ with non-positive Euler characteristic. Then:
\begin{itemize}
\item If $H \in (0,1)$, $\Sigma$ separates and $ \abs{\Sigma} = \dfrac{2\pi \abs{\chi(\Sigma)}}{1- H^2}$, $\chi(\Sigma ) < 0$, then there exists a totally geodesic minimal surface $\Sigma _m$ orthogonal to $\partial \amb$ and an isometry 
$$ F: \left( [0,  {\rm arctgh } \, H) \times \Sigma ,  dt^2 + \cosh ^2 \left( t \right)g_{-1} \right) \to \amb ,$$where $g_{-1}$ denotes the standard metric of constant curvature $-1$, such that 
\begin{itemize}
\item $F (0, \Sigma) = \Sigma $ and $F( {\rm arctgh } \, H , \Sigma) = \Sigma _{m}$, and
\item $ F (t, \Sigma ) = \Sigma _{t}$ is an embedded totally umbilic $H-$surface, $H={\rm tgh}( {\rm arctgh } \, H -t)$, orthogonal to $\partial \amb$ for all $t \in [0 ,  {\rm arctgh } \, H)$.
\end{itemize}

Moreover, if $\Sigma _m$ is non-orientable, a tubular neighborhood of $\Sigma _m$ is foliated by its equidistants $\Sigma _t$.



\item If $H=1$ and  $\chi (\Sigma) =0$, the mean convex side of $\amb \setminus \Sigma $, call this component $\mathcal U$, is isometric to $ [0, + \infty) \times \Sigma$ endowed with the product metric $g = dt^2 + e^{-2t}g_{e}$, where $g_e$ is the standard Euclidean metric of curvature $0$. That is, $\mathcal U$ is isometric to a cusp hyperbolic end $\mathcal P$ corresponding to Model 4 and $\Sigma$ is a slice.

\end{itemize}
\end{theorem}
\begin{proof}
As we did in Theorem \ref{Th:Rigidity}, let $N$ be the unit normal pointing to the mean convex side of $\amb \setminus \Sigma$ and consider the variation 
$$ F : [0,\epsilon _0 ) \times \Sigma  \to \amb  ,$$where 
$$ \epsilon _0 := {\rm sup}\set{\epsilon >0 \, : \, \, F : [0,\epsilon)\times \Sigma  \to \amb  \text{ is embedded and } \Sigma _t  \cap \partial \amb = \partial \Sigma _t  } >0. $$

As we observed in Theorem \ref{Th:Rigidity}, $\Sigma _t \subset \amb$, $\Sigma_t \cap \partial \amb = \partial \Sigma _t$ for all $t \in [0 ,\epsilon _0)$, where $\Sigma _t = F_t (\Sigma)$. Let $g = F^* \meta{}{}$ be the pullback metric of $\amb$ into $[0 ,\epsilon _0 ) \times \Sigma $ and, hence, we can write $g = dt^2 +g_t$. Since $\Sigma _0 = \Sigma$ achieves the equality in Lemma \ref{Lem:Area2}, then
\begin{itemize}
\item either $H=1$ and $\chi (\Sigma) =0$,
\item or $\Sigma _0$ is umbilic, $K_\Sigma = H^2 -1 < 0$ and $K_{sect} = -1$ along $\Sigma _0$, in other words, $g_0 := \frac{1}{\sqrt{1-H^2}} g_{-1}$, where $g_{-1}$ is the standard metric of constant sectional curvature $-1$.
\end{itemize}

Now, as we did in Theorem \ref{Th:Rigidity}, we study the behavior of the function $G(t)$ that satisfies
$$\int _{\Sigma _{t}} \Phi _{t}^{2}  \leq  G(t): =  2\pi \abs{ \chi (\Sigma)} - \abs{\Sigma _{t}} + \int _{\Sigma _{t}} H_{t}^{2} \text{ for all } t \in  [0 , \epsilon _0) .  $$

By the First Variation Formula and the First Variation of the Mean Curvature
$$ 2\dfrac{\partial H_t}{\partial t} ={\rm Ric}(N_t) + 2 H^2_t + 2 \Phi _t ^2 ,$$ we obtain
\begin{equation*}
G' (t) = \int _{\Sigma _t} 2 H_t \left( \frac{\partial H_t}{\partial t} - H_t ^2 +1 \right) = \int _{\Sigma _t}  H_t \left( 2 \Phi _t ^2 + {\rm Ric}(N_t) +2 \right)
\end{equation*}for all $t \in [0 , \epsilon _0 )$. Define 
$$ \bar \epsilon  = {\rm sup}\set{ \epsilon \in (0, \epsilon _0) \, :\,  H_{\epsilon}(p) \neq  0 \text{ for all }  p \in \Sigma _{\epsilon}} ,$$then, for any $\epsilon ' \in (0, \bar \epsilon )$, there exists $C \geq 0$ such that  $C \geq H_t (p) >0$ for all $(t,p) \in [0 ,  \epsilon ']\times \Sigma  $ and, therefore
\begin{equation}\label{Gprima}
G' (t) \leq  2 C \int _{\Sigma _t} \Phi _t ^2  \leq 2C \, G(t) \text{ for all } t \in [0 ,  \epsilon '] ,\end{equation}which, by Gronwall's Lemma, implies $ G(t) \leq 0 $ for all $ t \in [0 ,  \epsilon ']$, where we have used that $G(0)=0$ and hence $ G(t) = 0 $ for all $t \in [0,  \epsilon ']$. Hence, for all $t \in [0 ,  \epsilon '] $ we have
\begin{itemize}
\item [(b.1)] $\Sigma _t$ is umbilic,
\item [(b.2)]${\rm Ric}(N_t) + 2 =0$, which implies that the sectional curvature of any two plane orthogonal to $\Sigma$ is equals to $-1$,
\item [(b.3)]$ \dfrac{\partial H_t}{\partial t} = H_t^2 -1 $.
\end{itemize}

Now, we divide the proof into two cases. First, the case $H\in (0,1)$:

\begin{quote}
{\bf Case A: }{\it If $H \in (0,1)$, $\Sigma$ separates and $ \abs{\Sigma} = \dfrac{2\pi \abs{\chi(\Sigma)}}{1- H^2}$, $\chi(\Sigma ) < 0$, then there exists a totally geodesic minimal surface $\Sigma _m$ orthogonal to $\partial \amb$ and an isometry 
$$ F: \left( [0,  {\rm arctgh } \, H) \times \Sigma ,  dt^2 + \cosh ^2 \left( t \right)g_{-1} \right) \to \amb ,$$where $g_{-1}$ denotes the standard metric of constant curvature $-1$, such that 
\begin{itemize}
\item $F (0, \Sigma) = \Sigma $ and $F( {\rm arctgh } \, H , \Sigma) = \Sigma _{m}$, and
\item $ F (t, \Sigma ) = \Sigma _{t}$ is an embedded totally umbilic $H-$surface, $H={\rm tgh}( {\rm arctgh } \, H -t)$, orthogonal to $\partial \amb$ for all $t \in [0 ,  {\rm arctgh } \, H)$.
\end{itemize}

Moreover, if $\Sigma _m$ is non-orientable, a tubular neighborhood of $\Sigma _m$ is foliated by its equidistants $\Sigma _t$.}
\end{quote}
\begin{proof}[Proof of Case A]
On the one hand, by item (b.3) we have 
\begin{equation}\label{HNeg}
H_t (p) =  {\rm tgh}\left(  {\rm arctgh} \, H - t \right) \text{ for all } p \in \Sigma ,
\end{equation}that is, $\Sigma _t$ has constant mean curvature $H_t =  {\rm tgh}\left({\rm arctgh} \, H -t \right)$ for all $t \in  [0 ,  \epsilon ']$.

On the other hand, by the First Variation of the Area, we have 
$$ \dfrac{\partial }{\partial t} \abs{\Sigma _t} = -2 \, {\rm tgh}\left( {\rm arctgh} H -t \right) \abs{\Sigma _t} \text{ for all } t \in [0 ,  \epsilon '] , $$which implies that  
\begin{equation}\label{AreaNeg}
 \abs{\Sigma _t} =  2\pi \abs{\chi(\Sigma)}\cosh^2 \left( {\rm arctg} \, H -t\right) \text{ for all } t \in [0,  \epsilon '] .
\end{equation}

Thus, by the Gauss-Bonnet Formula substituting \eqref{HNeg} and \eqref{AreaNeg}, we have 
\begin{equation*}
2\pi \chi(\Sigma) = \int _{\Sigma _t} K_{\Sigma _t}  \leq (H_t^2 -1) \abs{\Sigma _t} = - 2 \pi \abs{\chi(\Sigma)},
\end{equation*}which tells us that the sectional curvature of the two plane tangent to $\Sigma _t$ is equal to $-1$ for all $t \in [0,  \epsilon '] $, and therefore
$$K_{\Sigma _t} = {\rm tgh}^2 ( {\rm arctgh} \, H -t ) -1= -\cosh^{-2} \left( {\rm arctgh} \, H  -t \right) \text{ for all } t \in [0,  \epsilon '] .$$

Now, we can argue with $\Sigma _{\epsilon '}$ and obtain $\epsilon '<  \epsilon ' _2 \leq \epsilon _0$ such that the above discussion is valid. Therefore, we can continue the process as far $H_t \geq 0$, that is, $\epsilon _0 = {\rm arctgh}\, H$ if $H\in (0,1)$. 

Set $\Sigma _{m} = F_{{\rm arctgh} \,H} (\Sigma ) $, then $\Sigma _{m}$ is a totally geodesic minimal surface such that $\abs{\Sigma _{m}} = 2\pi \abs{\chi (\Sigma)}$. 


Therefore, $ [0,  {\rm arctgh}\, H) \times \Sigma$ with the pullback metric $g = dt^2 + \cosh ^2 \left( t \right)g_{-1}$ is isometric to $\mathcal H ( {\rm arctgh}\, H)$ endowed with the hyperbolic metric given by Model 3. Denote by $I : \mathcal H ( {\rm arctgh}\, H) \to [0,  {\rm arctgh}\, H) \times \Sigma $ this global isometry. So, $\tilde F = F \circ I  :  \mathcal H ( {\rm arctgh}\, H)  \to \amb $ is a local isometry and $\tilde F$ is injective on $\Sigma$.


\begin{quote}
{\bf Claim 1: }{\it $ F $ is injective on $ [0, {\rm arctgh}\, H ) \times \Sigma $.}
\end{quote}
\begin{proof}[Proof of Claim 1]
Define 
$$ \bar \epsilon = {\rm sup}\set{\epsilon \in [0 , {\rm arctgh } \, H) \, : \,  F \text{ is injective on } [0,\epsilon ) \times \Sigma}. $$

Clearly, $\bar \epsilon > 0  $ since $F_{0} (\Sigma) = \Sigma $ is embedded, so injective. If $\bar \epsilon ={\rm arctgh} \, H$, we finish. We now prove F is injective on $ [0, {\rm arctgh}\, H ) \times \Sigma $.  Suppose the contrary.  Then there are distinct points $p,q \in \Sigma $ such that $1$ or $2$ holds:
\begin{enumerate}
\item $F(\bar \epsilon , p) = F(0, p ) \text{ or } F(0, q) $,
\item $F(\bar \epsilon , p) = F(\bar \epsilon , q) $.
\end{enumerate}

In case $1$, we can construct a closed curve $C$ meeting $\Sigma$ in exactly one point, which contradicts that $\Sigma $ separates $\amb$.  To construct $C$ when $F( \bar \epsilon , p) = F (0, p)$, take $C(t) = F(t, p)$, $0 \leq t  \leq  \bar \epsilon $. When $F(\bar \epsilon , p) = F (0, q)$, let $\Gamma (t)$ be a curve joining $q$ to $p$ on $\Sigma$, $0 \leq t  \leq  \bar \epsilon $.  Then define $C(t) = F(t, \Gamma (t))$, $0 \leq  t  \leq  \bar \epsilon $.  Thus case $1$ can not occur.

Consider case $2$, $F( \bar \epsilon ,p) = F(\bar \epsilon ,q)$. Let $U$ and $V$ be disjoint neighborhoods of $p$ and $q$, respectively, in $\set{\bar \epsilon} \times \Sigma $, such that $F(\bar \epsilon ,p)$ is in the intersection of $F(\bar \epsilon,U)$ and $F(\bar \epsilon,V)$. We  distinguish two possibilities here depending on the direction of the mean curvature vectors; $\vec{H}_{F(\bar \epsilon ,U)}$, $\vec{H}_{F(\bar \epsilon ,V)}$. We can assume that $U$ and $V$ are disjoint neighborhoods of $(\bar \epsilon , p)$ and $( \bar \epsilon  , q)$ on each of which $F$ is injective. For future use we extend $U$ and $V$ to neighborhoods $\tilde U $ and $ \tilde V$  in $ [0, \bar \epsilon ] \times \Sigma $ on each of which $F$ is injective.   

Suppose first that  $\vec{H}_{F(\bar \epsilon ,p)} =\vec{H}_{F(\bar \epsilon , q)}$. Then $F(\bar \epsilon , U)$ is on one side of $F(\bar \epsilon ,V)$ at $F(p,c)$ (where they are tangent) and they have the same mean curvature vector so $F(\bar \epsilon , U) = F( \bar \epsilon , V)$. Consequently, $\Sigma$ is a non trivial covering space of the embedded surface $F( \bar \epsilon , \Sigma)$.  But, $F(\tilde \epsilon , \Sigma)$, $\tilde \epsilon$ near $\bar \epsilon $, is a graph over $F(\bar \epsilon , \Sigma)$ in the trivial normal bundle of $F(\bar \epsilon , \Sigma)$ in $\amb$.  This is impossible when the covering space is non trivial.

Now suppose that $\vec{H}_{F(\bar \epsilon ,p)} = - \vec{H}_{F(\bar \epsilon , q)}$.  At $F(\bar \epsilon ,p)$, $F(\bar \epsilon , U)$ is strictly mean convex towards $\vec{H}_{F(\bar \epsilon , p)}$  and $F( \bar \epsilon , V)$ is strictly mean convex towards $-\vec{H}_{F(\bar \epsilon ,q)} $. $F(\tilde U)$ is on the mean concave side of $F( \bar \epsilon , U) $ near $F(\bar \epsilon , p) $. $F(\tilde V) $ is on the mean concave side of $F( \bar \epsilon ,V)$ near $F( \bar \epsilon ,p)$. So, any neighborhood of $F(\bar \epsilon , p)$ contains an open set on the mean concave side of $F(\bar \epsilon ,U)$ which is contained in the mean concave side of $F(\bar \epsilon ,V) $. Hence $F(\tilde U \setminus U)$ intersects $F(\tilde V \setminus V)$ which contradicts $F$ injective on $[0,\bar \epsilon) \times \Sigma $.

\end{proof}


We now know that $F$ is injective on $[0,{\rm arctgh} \, H) \times \Sigma $. The surface $\set{ {\rm arctgh} \, H }  \times \Sigma $ is a minimal surface so it may be the case that $\set{ {\rm arctgh} \, H } \times \Sigma $ is a two sheeted covering space of the non-orientable embedded surface
$F(\set{ {\rm arctgh} \, H } , \Sigma) := \Sigma _m $. Hence, a tubular neighborhood of $\Sigma _m$ foliated by its equidistants.

This finishes the proof of Claim A.
\end{proof}

Second, and last, when $H=1$:

\begin{quote}
{\bf Case B:} {\it If $H=1$, the mean convex side of $\amb \setminus \Sigma $ is isometric to  $ [0, + \infty) \times \Sigma$ endowed with the product metric $g = dt^2 + e^{-2t}g_{e}$.}
\end{quote}
\begin{proof}[Proof of Case B]
We follow \cite[Theorem 2]{LMazHRos}. In this case, the mean curvature of each $\Sigma _t$ is constant and it is equal to one, that is, $H_t = 1$ for all $t \in [ 0, \bar \epsilon )$. Thus, we can let $\bar \epsilon$ tend to $\epsilon _0$ and we get that $\Sigma _t$ has constant mean curvature $1$, it is umbilic, $K_\Sigma =0$ for all $t \in [0, \epsilon _0 )$ and, moreover, $[0, \epsilon _0 ) \times \Sigma$ endowed with the pullback metric $g$ has constant sectional curvature equals to $-1$ for any two plane. Also, one can see that $g = dt^2 + e^{-2t} g_e$, $g_e$ the standard Euclidean metric, and $\epsilon _0 = + \infty$. In this case, one can prove that, following \cite[Theorem 2]{LMazHRos}, the mean convex side of $\amb \setminus \Sigma $ is, in fact, isometric to $ [0, + \infty) \times \Sigma$ endowed with the product metric $g = dt^2 + e^{-2t}g_{e}$. 
\end{proof}

This concludes the proof of Theorem  \ref{Th:Rigidity2}.
\end{proof}

We remark that in the minimal case, i.e., $H=0$, by the First Variation Formula of the Mean curvature we obtain that $\abs{H_{t}(p)} >0$ for all $(t,p) \in [0,\epsilon ']$. Nevertheless, the mean curvature vector $\vec{H}(p)$ points in the opposite direction of the normal given by the flow $\frac{d}{dt}F_{t}(p)$. Hence, by our definition of the mean curvature function, $H_{t}$ is negative for $N_{t}$. In other words, in the minimal case, $N_{t} = \frac{d}{dt}F_{t}(p)$ and hence, \eqref{Gprima} is not satisfied. Recall that, in order to obtain \eqref{Gprima}, we needed that $H_t $ was positive for all $t$.

\subsection{Area estimate and rigidity of compact $H-$surfaces without boundary}

One can easily see that results stated above for compact $H-$surfaces with boundary can be extended to closed surfaces. First, the area estimate

\begin{lemma}\label{Lem:AreaCompact}
Let $\amb $ be a complete orientable Riemannian $3-$manifold with sectional curvatures $ K_{sect} \leq -1 $. Let $\Sigma \subset \amb$ be a closed oriented $H-$surface, $H \in [0,1)$, of genus $g(\Sigma ) \geq 2$. Then, 
\begin{equation*}
4\pi (g(\Sigma) -1) \geq ( 1- H^2 )\abs{\Sigma} .
\end{equation*}

Moreover, equality holds if, and only if, $\Sigma$ is umbilic, $K_\Sigma =  H^2-1$ and $K_{sect} \equiv -1$ along $\Sigma$.
\end{lemma}

In the case we get the equality in the area estimate above, we can classify the three-manifold
\begin{theorem}\label{Th:RigidityCompact}
Let $\amb $ be a complete orientable Riemannian $3-$manifold with sectional curvatures $K_{sect} \leq -1$. Assume that there exists a compact oriented embedded $H-$surface $\Sigma \subset \amb$, $H \in (0, 1)$, $\Sigma$ separates $\amb$, with negative Euler characteristic such that 
$$ \abs{\Sigma} = \dfrac{2\pi \abs{\chi(\Sigma)}}{1- H^2} ,$$then there exists a  totally geodesic minimal surface $\Sigma _m$ and an isometry 
$$ F: \left( [0,  {\rm arctgh } \, H ) \times \Sigma ,  dt^2 + \cosh ^2 \left( t \right)g_{-1} \right) \to \amb ,$$where $g_{-1}$ denotes the standard metric of constant curvature $-1$, such that 
\begin{itemize}
\item $F (0, \Sigma) = \Sigma $ and $F( {\rm arctgh } \, H , \Sigma) = \Sigma _{m}$, and
\item $ F (t, \Sigma ) = \Sigma _{t}$ is an embedded totally umbilic $H-$surface, $H={\rm tgh}( {\rm arctgh } \, H -t)$, for all $t \in [0 ,  {\rm arctgh } \, H)$.
\end{itemize}

Moreover, if $\Sigma _m$ is non-orientable, a tubular neighborhood of $\Sigma _m$ is foliated by its equidistants $\Sigma _t$.
\end{theorem}

The proof of the above result is completely analogous to Theorem \ref{Th:Rigidity2}.

\section*{Acknowledgments}

The first author, Jos\'{e} M. Espinar, is partially supported by Spanish MEC-FEDER Grant
MTM2013-43970-P; CNPq-Brazil Grants 405732/2013-9 and 14/2012 - Universal, Grant
302669/2011-6 - Produtividade; FAPERJ Grant 25/2014 - Jovem Cientista de Nosso Estado.

\end{document}